\documentclass[11pt]{amsart}
\usepackage{amsmath,amssymb}
\usepackage{fullpage}
\usepackage[dvips]{graphicx}
\usepackage{psfrag}  
\usepackage{amsfonts,amsthm,url,tabularx,array,enumerate,mathtools,verbatim,float} 
\usepackage{tikz} 

\usepackage{thmtools}
\usetikzlibrary{arrows} 
\def\COMMENT#1{}
\def\TASK#1{}

\newdimen\margin   
\def\textno#1&#2\par{%
    \margin=\hsize
    \advance\margin by -4\parindent
           \setbox1=\hbox{\sl#1}%
    \ifdim\wd1 < \margin
       $$\box1\eqno#2$$%
    \else
       \bigbreak
       \hbox to \hsize{\indent$\vcenter{\advance\hsize by -3\parindent
       \sl\noindent#1}\hfil#2$}%
       \bigbreak
    \fi}

\def\eps{\varepsilon}

\def\LL{\mathcal{L}}

\DeclarePairedDelimiter\floor{\lfloor}{\rfloor} 
\DeclarePairedDelimiter\ceil{\lceil}{\rceil}  

\newtheorem{thm}{Theorem}
\newtheorem{fact}[thm]{Fact}
\newtheorem{lemma}[thm]{Lemma}
\newtheorem{prop}[thm]{Proposition}
\newtheorem{cor}[thm]{Corollary}

\newtheorem{ques}[thm]{Question}
\newcommand{\msc}[1]{\begin{center}MSC2010: #1.\end{center}}
\declaretheoremstyle[notefont=\bfseries,notebraces={}{},%
    headpunct={},postheadspace=1em]{mystyle}
\declaretheorem[style=mystyle,numbered=no,name=Theorem]{thm-hand}
\declaretheorem[style=mystyle,numbered=no,name=Proposition]{prop-hand}

\allowdisplaybreaks

\title{On solution-free sets of integers II}
\author{Robert Hancock and  Andrew Treglown}
\thanks{The second author is supported by EPSRC grant EP/M016641/1.}
\begin{document}
\label{firstpage}

\begin{abstract} 
Given a linear equation $\LL$, a set $A \subseteq [n]$ is $\LL$-free if $A$ does not contain any `non-trivial' solutions to $\LL$.
We determine the precise size of the largest $\LL$-free subset of $[n]$
for several general classes of linear equations $\LL$ of the form 
$px+qy=rz$ for fixed $p,q,r \in \mathbb N$ where $p \geq q \geq r$.
Further, for \emph{all} such linear equations $\LL$, we give an upper bound on 
the number of maximal $\LL$-free subsets of $[n]$. 
In the case when $p=q\geq 2$ and $r=1$ this bound is exact up to an error term in the exponent. 
 We make use of container and removal lemmas of Green~\cite{G-R} to prove this result. Our results also extend to various linear equations with more than three variables.  
\end{abstract}
\date{\today}
\maketitle

\msc{11B75, 05C69}

\section{Introduction}
In this paper we study solution-free sets of integers, that is, sets that contain no solution to a given linear equation $\LL$. In particular, we investigate the size of the largest such subset of 
$[n]:=\{1, \dots, n\}$ and the number of maximal solution-free subsets of $[n]$.

More precisely,  consider a fixed linear equation $\LL$ of the form 
\begin{align*}
a_1x_1+\dots +a_k x_k =b
\end{align*}
where $a_1, \dots ,a_k ,b \in \mathbb Z$. If $b=0$ we say that $\LL$ is \emph{homogeneous}.
If $$\sum _{i \in [k]} a_i=b=0$$ then we say that $\LL$ is \emph{translation-invariant}. A solution $(x_1, \dots , x_k)$ to $\LL$ is said to be \emph{trivial} if $\LL$ is translation-invariant and if there exists a partition $P_1, \dots ,P_\ell$ of $[k]$ so that: 
\begin{itemize}
\item[(i)] $x_i=x_j$ for every $i,j$ in the same partition class $P_r$; 
\item[(ii)] For each $r \in [\ell]$, $\sum _{i \in P_r} a_i=0$. 
\end{itemize}
A set $A \subseteq [n]$ is \emph{$\LL$-free} if $A$ does not contain any non-trivial solutions to $\LL$. If the equation $\LL$ is clear from the context, then we simply say $A$ is \emph{solution-free}.

Many of the most famous results in combinatorial number theory concern solution-free sets. For example, Schur's theorem~\cite{schur} states that if $n$ is sufficiently large then  $[n]$ cannot be partitioned into  $r$ \emph{sum-free} sets (i.e. $\LL$ is $x+y=z$). 
Roth's theorem~\cite{roth} states that the largest \emph{progression-free} set (i.e. $\LL$ is $x_1+x_2=2x_3$) has size $o(n)$ whilst a classical result of  Erd\H{o}s and Tur\'an~\cite{erdos} determines, up to an error term, the size of the largest \emph{Sidon set} 
(i.e. $\LL$ is $x_1+x_2=x_3+x_4$).

\subsection{The size of the largest solution-free set}
As indicated above, a key question in the study of $\LL$-free sets is to establish the size $\mu_{\LL}(n)$ of the largest $\LL$-free subset of $[n]$. 
The study of this question for general $\LL$ was initiated by Ruzsa~\cite{ruzsa,ruzsa2}, who amongst other results, established that in general $\mu _{\LL} (n)=o(n)$ if $\LL$ is translation-invariant and  $\mu _{\LL} (n)=\Omega(n)$ otherwise. 
When $\LL$ is a homogeneous equation in two variables, the value of $\mu_{\LL}(n)$ is known exactly and an extremal $\LL$-free set can be found by greedy choice. See~\cite{hegarty} for further details.

For homogeneous linear equations in three variables, the picture is not as clear. First note we may assume without loss of generality that the equation is of the form $px+qy=rz$, where $p,q,r$ are fixed positive integers, and $\gcd(p,q,r)=1$. 

Now consider the following two natural candidates for extremal sets. Let $t:=\gcd(p,q)$ and let $a$ be the unique non-negative integer $0 \leq a < t$ such that $n-a$ is divisible by $t$. The interval $$I_n:=\left[ \left \lfloor \frac{r(n-a)}{p+q} \right \rfloor +1,n\right]$$ is $\LL$-free. To see this observe that since $\gcd(p,q,r)=1$ and $\gcd(p,q)=t$, any solution $(x,y,z)$ to $\LL$ with $x,y,z \in I_n$ must have $z$ divisible by $t$. Since $px+qy>r(n-a)$, $z$ must lie in $[n-a+1,n]$; however then $z$ is not divisible by $t$ and so $I_n$ is $\LL$-free. Note that when $r=1$, $I_n:=[\floor{r(n-a)/(p+q)}+1,n]=[\floor{rn/(p+q)}+1,n]$, though this does not hold in general when $r>1$.\COMMENT{ For example, take $5x+5y=3z$ and $n=34$.} 
Notice also that $I_n$ is only a candidate for an extremal set if $r$ is `small'. Indeed,
 if $r>p+q$ and $n$ is sufficiently large then $I_n=\emptyset$. The set $$T_n:=\{x \in [n]: x \not \equiv 0 \text{ mod } t\}$$ is also $\LL$-free: note that in any solution $(x,y,z)$ to $\LL$, $z$ must be divisible by $t$ since $\gcd(r,t)=1$. But $T_n$ contains no elements divisible by $t$.

This raises the following question.
\begin{ques}\label{maxsizeq}
For which $\LL$ do we have $\mu_{\LL}(n)=\max\{|I_n|,|T_n|\}$?
\end{ques}
When $\LL$ is $x+y=z$ it is easy to see that $\mu_{\LL}(n)=\lceil n/2 \rceil$ and the interval $I_n=[\floor{n/2}+1,n]$ is an extremal set of this size. Recently, the authors~\cite{HT1}  established that if $\LL$ is the equation $px+qy=z$ with $p,q \in \mathbb{N}$, $p\geq 2$, then $\mu_{\LL}(n)=n-\floor{n/(p+q)}$ for sufficiently large $n$. Again this bound is attained by the interval $I_n$. Our first result of this paper determines a further class of equations (of the form $px+qy=rz$) for which $I_n$ or $T_n$ gives an $\LL$-free set of maximum size.

\begin{thm}\label{Mu4}
Let $\LL$ denote the equation $px+qy=rz$ where $p\geq q \geq r$ and $p,q,r$ are fixed positive integers satisfying $\gcd(p,q,r)=1$. Let $t:=\gcd(p,q)$. Write $r_1:=p/t$ and $r_2:=q/t$.
\begin{enumerate}[(i)]
\item{If $q$ divides $p$ and $p+q \leq rq$ then $\mu_{\LL}(n)=\ceil{(q-1)n/q}$;}
\item{If $q$ divides $p$ and $p+q \geq rq$ then $\mu_{\LL}(n)=\ceil{(p+q-r)(n-a)/(p+q)}+a$ where $a$ is the unique non-negative integer $0 \leq a < q$ such that $n-a$ is divisible by $q$;}
\item{If $q$ does not divide $p$, $t>1$ and $$r>(r_1r_2-r_1-r_2+4)r_2\bigg(r_1+1+\frac{r_2-1}{r_1^2+(r_1-1)(r_2-1)}\bigg)$$ then $\mu_{\LL}(n)=\ceil{(t-1)n/t}$.}
\end{enumerate}
\end{thm}
Theorem~\ref{Mu4}(ii) was already proven (for large enough $n$) in~\cite{HT1} in the special case when $r=1$. (Note though that our work in~\cite{HT1} determines $\mu _{\LL}(n)$ for many equations $\LL$ not covered by Theorem~\ref{Mu4}.) Previously,
Hegarty~\cite{hegarty} proved Theorem~\ref{Mu4}(i) in the case when $p = q$. In Section~\ref{2LSFS} we also give a generalisation  of Theorem~\ref{Mu4} concerning some linear equations with more variables (see Corollary~\ref{Mu5}).


Notice that in the case when $q$ divides $p$, Theorem~\ref{Mu4} gives a dichotomy for the value of $\mu_{\LL}(n)$: when $p+q \leq rq$ the set $T_n$ is a largest $\LL$-free subset of $[n]$, whilst when $p+q \geq rq$ the interval $I_n$ is a largest $\LL$-free subset of $[n]$. Theorem~\ref{Mu4} does not provide us with as much information for the case when $q$ does not divide $p$; note though it is not true that a similar dichotomy occurs in this case. Take the equation $3x+2y=2z$; here we have $|I_n| \approx 3n/5$ and $|T_n|=0$. However the set $A_n:=\{x \in [n]: x \not \equiv 0 $ mod $ 2$ or $x>2n/3\}$ has size $|A_n| \approx 2n/3$ and is $\LL$-free, since any solution $(x,y,z)$ to $\LL$ must have that $x$ is even and $x \leq 2n/3$. It would be very interesting to fully resolve the case where $p \geq q \geq r$ and $q$ does not divide $p$.

For equations $px+qy=rz$ where $r$ is bigger than $p,q$, there are a range of cases where an extremal set is known and it is neither $I_n$ nor $T_n$; see~\cite{baltz, dil, hegarty} for these, and also other results on the size of the largest $\LL$-free subset of $[n]$ for various $\LL$.


\subsection{The number of maximal solution-free sets}
Given a linear equation $\LL$, write $f(n,\LL)$  for the number of $\LL$-free subsets of $[n]$. Observe that all possible subsets of an $\LL$-free set are also $\LL$-free, and so $f(n,\LL) \geq 2^{\mu_{\LL}(n)}$. In fact, in general this trivial lower bound is not too far from the precise value of $f(n,\LL)$.
Indeed, Green~\cite{G-R} showed that given a homogeneous linear equation $\LL$, then $f(n, \LL)=2^{\mu_{\LL} (n)+o(n)}$ (where here the $o(n)$ may depend on $\LL$).
(Recall though, in the case when $\LL$ is translation-invariant $\mu_{\LL} (n)=o(n)$, so Green's theorem only tells us that $f(n, \LL)=2^{o(n)}$ for such $\LL$.) 
In the case of sum-free sets, Green~\cite{G-CE} and independently Sapozhenko~\cite{sap},  showed that there are constants $C_1$ and $C_2$ such that  $f(n, \LL)=(C_i+o(1))2^{n/2}$ for all $n \equiv i \mod 2$. 
This resolved a conjecture of Cameron and Erd\H{o}s~\cite{cam1}.

Far less is known about the number  $f_{\max} (n, \LL)$ of maximal $\LL$-free subsets of $[n]$. (We say that $A\subseteq [n]$ is a \emph{maximal $\LL$-free subset of $[n]$} if it is $\LL$-free and it is not properly contained in another $\LL$-free subset of $[n]$.) In the case when $\LL$ is $x+y=z$, Cameron and Erd\H{o}s~\cite{CE}  asked whether $f_{\max} (n,\LL) \leq f(n,\LL)/ 2^{\eps n}$ for some constant $\eps >0$;
 a few years later \L{u}czak and Schoen~\cite{ls} confirmed this to be true. 
After further progress on the problem~\cite{wolf, BLST}, Balogh, Liu, Sharifzadeh and Treglown~\cite{BLST2} proved the following sharp result for maximal sum-free sets: For each $1\leq i \leq 4$, there is a constant $C_i$ such that, given any $n\equiv i \mod 4$, $f_{\max}(n,\LL)=(C_i+o(1)) 2^{n/4}$.

 For other linear equations $\LL$, it is also natural to ask whether there are significantly fewer maximal $\LL$-free subsets of $[n]$ than there are $\LL$-free subsets.
In \cite{HT1} we showed that this is the case for all non-translation-invariant three-variable homogeneous equations $\LL$. In particular in this case $f_{\max}(n,\LL) \leq 3^{(\mu_{\LL}(n)-\mu^*_{\LL}(n))/3+o(n)}$ where $\mu^*_{\LL}(n)$ denotes the number of elements in $[n]$ which do not lie in \emph{any} non-trivial solution to $\LL$ that only consists of elements from $[n]$. We also gave other  upper bounds on $f_{\max}(n,\LL)$ for equations of the form $px+qy=z$ for fixed $p,q \in \mathbb N$ which in most cases yielded better bounds (see~\cite{HT1}).
 In this paper we prove the following result.
\begin{thm}\label{MainT1}
Let $\LL$ denote the equation $px+qy=rz$ where $p \geq q \geq r$ and $p,q,r$ are fixed positive integers satisfying $\gcd(p,q,r)=1$. Let $t:=\gcd(p,q)$. Then $$f_{\max}(n,\LL) \leq 2^{\frac{Crn}{q}+o(n)} \text{ where } C:=1-\frac{t}{p+q}\Big(\frac{p^2+(p-t)(q-t)}{p^2} \Big).$$
\end{thm}

For a wide class of equations $\LL$  this is the current best known upper bound on $f_{\max}(n,\LL)$; see the appendix for more details. 
In the case when $p=q\geq 2$ and $r=1$, the upper bound given by Theorem~\ref{MainT1} is actually exact up to the error term in the exponent.

\begin{thm}\label{qxqyz}
Let $\LL$ denote the equation $qx+qy=z$ where $q\geq 2$ is an integer. Then $$f_{\max}(n,\LL)=2^{n/2q+o(n)}.$$
\end{thm}
In Section~\ref{2MSFS} we will also generalise Theorem~\ref{MainT1} to consider some linear equations with more variables (see Theorem~\ref{MainT1Ext}).

For the proof of both Theorems~\ref{Mu4} and~\ref{MainT1}, a simple but crucial tool is a result (Lemma~\ref{GM1}) which ensures a certain auxiliary graph contains a large collection of disjoint edges.
 To prove Theorem~\ref{MainT1} we also make use of container and removal lemmas of Green~\cite{G-R} (see Section~\ref{21}).

In the next section  we collect together a number of useful tools and lemmas. We prove our results on the size of the largest solution-free subset of $[n]$, and on the number of maximal solution-free subsets of $[n]$, in Section~\ref{2LSFS} and~\ref{2MSFS} respectively.

\section{Containers, link hypergraphs and the main lemmas}\label{2tools}
\subsection{Container and removal lemmas}\label{21}

Observe that we can formulate the study of $\LL$-free sets in terms of independent sets in hypergraphs. Let $H$ denote the hypergraph with vertex set $[n]$ and edges corresponding to non-trivial solutions to $\LL$. Then an independent set in $H$ is precisely an $\LL$-free set. The method of containers roughly states that for certain (hyper)graphs $G$, the independent sets of $G$ lie only in a small number of subsets of $V(G)$ called \emph{containers}, where each container is an `almost independent set'. In general, the method of containers has had a wide number of applications to combinatorics and other areas;
the method for \emph{graphs} was first developed by Kleitman and Winston~\cite{kw1, kw2}. More recently, the \emph{hypergraph} container method was developed by
  Balogh, Morris and Samotij~\cite{BMS} and independently by Saxton and Thomason~\cite{ST}. In this section we introduce a container lemma of Green~\cite{G-R} for $\LL$-free sets. 


 
Lemma~\ref{L1Ext}(i)--(iii) is stated explicitly in Proposition 9.1 of~\cite{G-R}. Lemma~\ref{L1Ext}(iv) follows as an immediate consequence of Lemma~\ref{L1Ext}(i) and Lemma~\ref{L2Ext} below.

\begin{lemma}\label{L1Ext}
\cite{G-R} 
Fix a $k$-variable homogeneous linear equation $\LL$. There exists a family $\mathcal{F}$ of subsets of $[n]$ with the following properties:

\begin{enumerate}[(i)]
  \item{Every $F \in \mathcal{F}$ has at most $o(n^{k-1})$ solutions to $\LL$.}
  \item{If $S \subseteq [n]$ is $\LL$-free, then $S$ is a subset of some $F \in \mathcal{F}$.}
  \item{$|\mathcal{F}|=2^{o(n)}$.}
  \item{Every $F \in \mathcal{F}$ has size at most $\mu_\LL(n)+o(n)$.}
\end{enumerate}
\end{lemma}

We call the elements of $\mathcal{F}$ \emph{containers}. Observe that Lemma~\ref{L1Ext}(iv) gives a bound on the size of the containers in terms of $\mu_\LL(n)$, even in the case when $\mu_\LL(n)$ is not known.


The following removal lemma is a special case of a result of Green (Theorem 1.5 in~\cite{G-R}). This result was also generalised to systems of linear equations by Kr\'al', Serra and Vena (Theorem 2 in ~\cite{ksv2}). 
\begin{lemma}\label{L2Ext}
\cite{G-R}  
Fix a $k$-variable homogeneous linear equation $\LL$. Suppose that $A \subseteq [n]$ is a set containing $o(n^{k-1})$ solutions to $\LL$. Then there exist $B$ and $C$ such that $A=B \cup C$ where $B$ is $\LL$-free and $|C|=o(n)$.
\end{lemma}



We will use both the above results to obtain bounds on the number of maximal $\LL$-free sets.

\subsection{Link hypergraphs}
One can turn the problem of counting the number of maximal $\LL$-free subsets of $[n]$ into one of counting maximal independent sets in an auxiliary graph. Similar techniques were used in~\cite{wolf, BLST, BLST2, HT1}, and in the graph setting in~\cite{sar, BLPS}. To be more precise let $B$ and $S$ be disjoint subsets of $[n]$ and fix a three-variable linear equation $\LL$. The \emph{link graph $L_S[B]$ of $S$ on $B$} has vertex set $B$, and an edge set consisting of the following two types of edges:

\begin{enumerate}[(i)]
\item{ Two vertices $x$ and $y$ are adjacent if there exists an element $z \in S$ such that $\{x,y,z\}$ is an $\LL$-triple;}
\item{ There is a loop at a vertex $x$ if there exists an element $z \in S$ or elements $z,z' \in S$ such that  $\{x,x,z\}$ or $\{x,z,z'\}$ is an $\LL$-triple.}
\end{enumerate}
Here by  \emph{$\LL$-triple} we simply mean a multiset $\{x,y,z\}$ which forms a solution to $\LL$.

Consider the following generalisation of a link graph $L_S[B]$ to that of a \emph{link hypergraph}: Let $B$ and $S$ be disjoint subsets of $[n]$ and let $\LL$ denote the equation $p_1 x_1 + \cdots + p_k x_k=0$ where $p_1,\dots,p_k$ are fixed non-zero integers. The \emph{link hypergraph} $L_S[B]$ of $S$ on $B$ has vertex set $B$; It has an edge set consisting of hyperedges between $s\leq k$ distinct vertices $v_1,\dots,v_s$ of $B$, whenever there is a solution $(x_1,\dots,x_k)$ to $\LL$ in which $\{x_1,\dots,x_k\} \subseteq S \cup \{v_1,\dots,v_s\}$ and $\{v_1,\dots,v_s\} \subseteq \{x_1,\dots,x_k\}$. In  this definition one could have edges corresponding to trivial solutions. However in our applications, since we only consider non-translation-invariant equations, there are no trivial solutions.

The link graph lemmas used by the authors in~\cite{HT1} (Lemmas 12 and 15) can easily be extended to the hypergraph case.

\begin{lemma}\label{L4Ext}
Let $\LL$ denote a non-translation-invariant linear equation. Suppose that $B,S$ are disjoint $\LL$-free subsets of $[n]$. If $I \subseteq B$ is such that $S \cup I$ is a maximal $\LL$-free subset of $[n]$, then $I$ is a maximal independent set in the link hypergraph $L_{S}[B]$.\qed
\end{lemma}

Let MIS$(G)$ denote the number of maximal independent sets in $G$. The above result can be used in conjunction with the container lemma as follows. Let $F=A \cup B$ be a container as in Lemma~\ref{L1Ext}
where $|A|=o(n)$ and $B$ is $\LL$-free. Observe that any maximal $\LL$-free subset of $[n]$ in $F$ can be found by first selecting an $\LL$-free subset $S \subseteq A$, and then extending $S$ in $B$. Then the number of extensions of $S$ in $B$ is bounded by MIS$(L_S[B])$ by Lemma~\ref{L4Ext}. 

We can also use link graphs to obtain lower bounds.

\begin{lemma}\label{L6Ext}
Let $\LL$ denote a non-translation-invariant linear equation. Suppose that $B,S$ are disjoint $\LL$-free subsets of $[n]$. Let $H$ be an induced subgraph of the link graph $L_S[B]$. Then $f_{\max}(n,\LL) \geq \,$\rm MIS\em $(H)$. \qed
\end{lemma}

\subsection{The main lemmas}

Here we use a specific link graph as a means to bound the  number of elements in a solution-free subset of $[n]$.

Let $\LL$ denote the equation $px+qy=rz$ where $p \geq q \geq r$ and $p,q,r$ are fixed positive integers satisfying $\gcd(p,q,r)=1$. Let $t:=\gcd(p,q)$ and write $r_1:=p/t$, $r_2:=q/t$. Fix $M \in [n]$ such that $M$ is divisible by $t$. We define the graph $G_M$ to have vertex set $[\ceil{rM/q}-1]$ and an edge between $x$ and $y$ whenever $px+qy=rM$. 

\begin{lemma}\label{GM1}
The graph $G_M$ contains a collection $E$ of vertex-disjoint edges where $$|E|=\left \lfloor \frac{rM}{r_2(p+q)} \right \rfloor +(r_1 r_2-r_1-r_2+1) \left \lfloor \left \lfloor \frac{rM}{r_1(p+q)}-\frac{1}{r_2} \right \rfloor \frac{1}{r_1 r_2} \right \rfloor$$ and at most one edge in $E$ is a loop.
\end{lemma}

\begin{proof}
All edges in $G_M$ are pairs of the form $\{s,(rM-sp)/q\}$ for some $s \in \mathbb{N}$ since $ps+q(rM-sp)/q=rM$. Since $p=r_1 t$ and $q=r_2 t$ where $r_1$ and $r_2$ are coprime, for a fixed integer $s$ precisely one element in $\{(rM-(s-j)p)/q : 0 \leq j < r_2 \}$ is an integer. (Note here we are using that $M$ is divisible by $t$.) In other words there exists a unique $x \in \mathbb{N}$, $1 \leq x \leq r_2$ such that $(rM-xp)/q$ is an integer, and all edges in $G_M$ are of the form $\{x+ar_2,(rM-xp)/q-ar_1\}$ for some non-negative integer $a$. In particular there is an edge provided $a$ satisfies $(rM-xp)/q-ar_1 \in \mathbb{N}$. 

Write $y:=(rM-xp)/q$. Note that if $x+ar_2 \leq rM/(p+q)$ then $$y-ar_1 = \frac{rM-(x+ar_2)p}{q} \geq \frac{rM-prM/(p+q)}{q} = \frac{rM}{p+q}.$$ Hence there are $\floor{rM/(r_2(p+q))}$ distinct edges in $G_M$ of the form $\{x+ar_2,y-ar_1\}$ with $x+ar_2 \leq rM/(p+q) \leq y-ar_1$. Note that one of these edges may be a loop. (This will be at $rM/(p+q)$ in the case when $rM/(p+q) \in \mathbb{N}$.) Call this collection of edges $E_1$. Our next aim is to find an additional collection $E_2$ of edges in $G_M$ that is vertex-disjoint from $E_1$.

Note that $x+ar_2 \equiv x$ mod $ r_2$ and $y-ar_1 \equiv y$ mod $ r_1$. Also $p(rM/p)+q(0)=rM$ and $\ceil{rM/p} \leq \ceil{rM/q}$, hence there are at least $$\left \lfloor \bigg( \left \lceil \frac{rM}{p} \right \rceil - 1 - \left \lfloor \frac{rM}{p+q} \right \rfloor \bigg) /r_2 \right \rfloor \geq \left \lfloor \bigg( \frac{rM}{p} - \frac{rM}{p+q}-1 \bigg)/r_2 \right \rfloor =\left \lfloor \frac{rM}{r_1(p+q)} - \frac{1}{r_2} \right \rfloor$$ edges in $G_M$ of the form $\{x+ar_2,y-ar_1\}$ with $x+ar_2 > rM/(p+q)$. Consider a set of $r_1 r_2$ edges $\{ \{x+ar_2+br_2,y-ar_1-br_1\} : 0 \leq b < r_1 r_2\}$ for a fixed $a$. Since $r_1$ and $r_2$ are coprime, precisely $r_2$ of these edges (1 in $r_1$ of them) have $x+ar_2+br_2 \equiv y$ mod $r_1$, and precisely $r_1$ of these edges have $y-ar_1-br_1 \equiv x$ mod $r_2$. (Also, precisely 1 edge satisfies both.) In all other cases since $x+ar_2+br_2 \not \equiv y$ mod $r_1$ and $y-ar_1-br_1 \not \equiv x$ mod $r_2$, the edge $\{x+ar_2+br_2,y-ar_1-br_1\}$ is vertex-disjoint from $E_1$. Hence we obtain a set $E_2$ of at least $(r_1 r_2-r_1-r_2+1) \floor{\floor{rM/(r_1(p+q))-1/r_2}/(r_1 r_2)}$ additional distinct edges. Thus $E:=E_1 \cup E_2$ is our desired set.
\end{proof}

Observe that the graph $G_M$ is a link graph $L_S[B]$, where $S:=\{M\}$ and $B:=[\ceil{rM/q}-1]$. If we wish to extend a solution-free set $S$ into a solution-free subset of $S \cup B$, then we must pick an independent set in $L_S[B]$. Similarly here if we wish to obtain a solution-free subset of $[n]$ which contains $M$ divisible by $t$, then we must pick an independent set in $G_M$. This is the idea behind the following key lemma, which allows us to bound the number of elements in such an $\LL$-free set. 

\begin{lemma}\label{MainL1}
Let $\LL$ denote the equation $px+qy=rz$ where $p \geq q \geq r$ and $p,q,r$ are fixed positive integers satisfying $\gcd(p,q,r)=1$. Let $t:=\gcd(p,q)$ and write $r_1:=p/t$ and $r_2:=q/t$. Let $S$ be an $\LL$-free subset of $[n]$. If $M \in S$ is divisible by $t$, then $S$ contains at most 
$$\left \lceil \frac{rM}{q} \right \rceil -1 - \left \lfloor \frac{rM}{r_2(p+q)} \right \rfloor -(r_1 r_2-r_1-r_2+1) \left \lfloor \left \lfloor \frac{rM}{r_1(p+q)}-\frac{1}{r_2} \right \rfloor \frac{1}{r_1 r_2} \right \rfloor$$ elements from $[\ceil{rM/q}-1]$.
\end{lemma}

\begin{proof}
Consider the graph $G_M$ and observe that its edges correspond to $\LL$-triples: since $p \geq q \geq r$ there is an edge between $x$ and $y$ precisely when $\{x,y,M\}$ is an $\LL$-triple. Hence if $I \subseteq V(G_M)$ is such that $I \cup \{M\}$ is an $\LL$-free subset of $[n]$ then $I$ is an independent set in $G_M$. As a consequence if we find a set of vertex-disjoint edges in $G_M$ of size $J$, then $S$ contains at most $\ceil{rM/q}-1-J$ elements from $[\ceil{rM/q}-1]$. The result then follows by applying Lemma~\ref{GM1}. 
\end{proof} 

First note that if $\LL$ denotes the equation $x+y=z$, then in Lemma~\ref{MainL1} we are simply saying that if a sum-free set $S$ contains $M$, then it cannot contain both $1$ and $M-1$, it cannot contain both $2$ and $M-2$, and so on. So in a sense this lemma is a generalisation of the proof that sum-free subsets of $[n]$ cannot contain more than $\ceil{n/2}$ elements.

Let $\LL$ denote the equation $px+qy=rz$ where $p \geq q \geq r$ and $p,q,r$ are fixed positive integers satisfying $\gcd(p,q,r)=1$ and let $t:=\gcd(p,q)$. Recall that $T_n:=\{x \in [n]: x \not \equiv 0 $ mod $ t\}$ is $\LL$-free. Lemma~\ref{MainL1} roughly implies that every $\LL$-free subset of $[n]$ must have `not too many small elements' or must `look like' $T_n$. Clearly this lemma gives rise to an upper bound on the size of the largest $\LL$-free subset of $[n]$. In Section~\ref{2MSFS} we also show that  this lemma can be used to obtain an upper bound on the number of maximal $\LL$-free subsets of $[n]$.






The following simple proposition allows us to extend our results for linear equations with three variables to linear equations with more than three variables. 


\begin{prop}\label{Mu1} 
Let $\LL_1$ denote the equation $p_1x_1+\dots +p_k x_k =b$ where $p_1, \dots ,p_k ,b \in \mathbb{Z}$ and let $\LL_2$ denote the equation $(p_1+p_2)x_1+p_3 x_2 +\dots +p_k x_{k-1} =b$. Then $\mu_{\LL_1}(n) \leq \mu_{\LL_2}(n)$.
\end{prop}

\begin{proof}
If $(p_1+p_2)x_1+p_3 x_2 +\dots +p_k x_{k-1} =b$ for some $x_i \in [n]$, $1 \leq i \leq k-1$, then $p_1 x_1 +p_2 x_1 +p_3 x_2 +\dots +p_k x_{k-1} =b$. Hence any solution to $\LL_2$ in $[n]$ gives rise to a solution to $\LL_1$ in $[n]$. So if $A \subseteq [n]$ is $\LL_1$-free, then $A$ is also $\LL_2$-free. Hence the size of the largest $\LL_2$-free set is at least the size of the largest $\LL_1$-free set.
\end{proof}
We will also make use of the following  trivial fact.

\begin{fact}\label{Mu6}
Suppose $\LL_1$ is a linear equation and $\LL_2$ is a positive integer multiple of $\LL_1$. Then the set of $\LL_1$-free subsets of $[n]$ is precisely the set of $\LL_2$-free subsets of $[n]$. In particular $\mu_{\LL_1}(n)=\mu_{\LL_2}(n)$, $f(n,\LL_1)=f(n,\LL_2)$ and $f_{\max}(n,\LL_1)=f_{\max}(n,\LL_2)$. 
\end{fact}


The two results above  allow us to extend the use of Lemma~\ref{MainL1} to equations with more than three variables. 

\begin{lemma}\label{MainL1Ext}
Let $\LL$ denote the equation $p_1 x_1 + \cdots  + p_k x_k=0$ where $p_i \in \mathbb{Z}$. Suppose there is a partition of the $p_i$ into three non-empty parts $P_1$, $P_2$ and $P_3$ where $p':=\sum_{p_j \in P_1} p_j$, $q':=\sum_{p_j \in P_2} p_j$ and $r':=-\sum_{p_j \in P_3} p_j$ satisfy $p' \geq q' \geq r' \geq 1$. Let $t':=\gcd(p',q',r')$ and write $p:=p'/t'$, $q:=q'/t'$ and $r:=r'/t'$. Let $t:=\gcd(p,q)$ and write $r_1:=p/t$ and $r_2:=q/t$. Let $S$ be an $\LL$-free subset of $[n]$. If $M \in S$ is divisible by $t$, then $S$ contains at most $$\left \lceil \frac{rM}{q} \right \rceil - 1 - \left \lfloor \frac{rM}{r_2(p+q)} \right \rfloor -(r_1 r_2-r_1-r_2+1) \left \lfloor \left \lfloor \frac{rM}{r_1(p+q)}-\frac{1}{r_2} \right \rfloor \frac{1}{r_1 r_2} \right \rfloor$$ elements from $[\ceil{rM/q}-1]$.
\end{lemma}

\begin{proof}
Let $\LL'$ denote the equation $px+qy=rz$. Now observe by repeatedly applying Proposition~\ref{Mu1} and Fact~\ref{Mu6} that any $\LL$-free set is also an $\LL'$-free set. Hence $S$ must be $\LL'$-free and so we simply apply Lemma~\ref{MainL1}.
\end{proof}

This bounds the number of `small elements' in solution-free sets for equations with more than three variables, and in Theorem~\ref{MainT1Ext} we will use this lemma to obtain a result for the number of maximal solution-free sets.

\section{The size of the largest solution-free set}\label{2LSFS}

The aim of this section is to use our results from the previous section to obtain bounds on $\mu_{\LL}(n)$ for linear equations $\LL$ of the form $px+qy=rz$ with $p\geq q \geq r$ positive integers and also linear equations with more than three variables. As previously mentioned we can use Lemma~\ref{MainL1} to obtain a bound on the size of a solution-free set. 

\begin{cor}\label{Mu3}
Let $\LL$ denote the equation $px+qy=rz$ where $p\geq q \geq r$ and $p,q,r$ are fixed positive integers satisfying $\gcd(p,q,r)=1$. Let $S$ be an $\LL$-free subset of $[n]$ and suppose $M$ is the largest element of $S$ divisible by $t:=\gcd(p,q)$. Write $r_1:=p/t$ and $r_2:=q/t$. Then $$|S| \leq M - \left \lfloor \frac{rM}{r_2(p+q)} \right \rfloor -(r_1 r_2 -r_1-r_2+1) \left \lfloor \left \lfloor \frac{rM}{r_1(p+q)} - \frac{1}{r_2} \right \rfloor \frac{1}{r_1 r_2}\right \rfloor + \left \lceil \frac{(n-M)(t-1)}{t} \right \rceil.$$
\end{cor}

\begin{proof}
By Lemma~\ref{MainL1}, $S$ contains at most $\ceil{rM/q}-1-\floor{rM/(r_2(p+q))}-(r_1 r_2 -r_1-r_2+1)\floor{\floor{rM/(r_1(p+q))-1/r_2}/(r_1 r_2)}$ elements from $[\ceil{rM/q}-1]$. It also cannot contain any element larger than $M$ and divisible by $t$.
\end{proof}

Note in the statement of Corollary~\ref{Mu3} we are implicitly assuming that $M$ exists. If it does not then $|S| \leq \ceil{n(t-1)/t}$.

We are now ready to prove Theorem~\ref{Mu4}, which determines $\mu_{\LL}(n)$ for a wide class of equations of the form $px+qy=rz$ where $p\geq q \geq r$ and $p,q,r$ are fixed positive integers.


\begin{thm-hand}[\ref{Mu4}.] 
{\it
Let $\LL$ denote the equation $px+qy=rz$ where $p\geq q \geq r$ and $p,q,r$ are fixed positive integers satisfying $\gcd(p,q,r)=1$. Let $t:=\gcd(p,q)$. Write $r_1:=p/t$ and $r_2:=q/t$.
\begin{enumerate}[(i)]
\item{If $q$ divides $p$ and $p+q \leq rq$ then $\mu_{\LL}(n)=\ceil{(q-1)n/q}$;}
\item{If $q$ divides $p$ and $p+q \geq rq$ then $\mu_{\LL}(n)=\ceil{(p+q-r)(n-a)/(p+q)}+a$ where $a$ is the unique non-negative integer $0 \leq a < q$ such that $n-a$ is divisible by $q$;}
\item{If $q$ does not divide $p$, $t>1$ and $$r>(r_1r_2-r_1-r_2+4)r_2\bigg(r_1+1+\frac{r_2-1}{r_1^2+(r_1-1)(r_2-1)}\bigg)$$ then $\mu_{\LL}(n)=\ceil{(t-1)n/t}$.}
\end{enumerate}}
\end{thm-hand}

\begin{proof}
Let $S$ be an $\LL$-free subset of $[n]$ and suppose $M$ is the largest element of $S$ divisible by $t$. If $S$ does not contain an element divisible by $t$, set $M:=0$. If $q$ divides $p$ then $t=q$ and $r_2=1$ and hence by Corollary~\ref{Mu3} we have 
\begin{align}\label{sizeeq}
|S| \leq \left \lceil \frac{(p+q-r)M}{p+q} \right \rceil + \left \lceil \frac{(n-M)(q-1)}{q} \right \rceil .
\end{align} 
(This is true even in the case $M=0$.) 

If $p+q \leq rq$ then $|S| \leq \ceil{(q-1)M/q}+\ceil{(n-M)(q-1)/q}=\ceil{n(q-1)/q}$ since $M$ is divisible by $q$. Observe that the set $T_n:=\{x \in [n]: x \not \equiv 0 $ mod $ t\}$ is an $\LL$-free set obtaining this size, and so this proves (i).

For (ii) we will show that (\ref{sizeeq}) is an increasing function of $M$ (when restricted to running through $M$ divisible by $t$) and hence it will be maximised by taking $M=n-a$. Then $|S| \leq \ceil{(p+q-r)(n-a)/(p+q)}+a$. Observe that the interval $I_n:=[\floor{r(n-a)/(p+q)}+1,n]$ is an $\LL$-free set obtaining this size and so this proves (ii), provided (\ref{sizeeq}) is an increasing function of $M$.

Since $M$ must be divisible by $t=q$, write $M':=M/q$ and so (\ref{sizeeq}) can be written as $$\left \lceil \frac{((r_1+1)q-r)M'}{r_1+1} \right \rceil + \left \lceil \frac{n(q-1)}{q} \right \rceil - M'(q-1)= M'+\left \lceil \frac{-rM'}{r_1+1} \right \rceil + \left \lceil \frac{n(q-1)}{q} \right \rceil.$$

Now observe that the difference between successive terms $M'$ and $M'+1$ is given by $$M'+1+\left \lceil \frac{-r(M'+1)}{r_1+1} \right \rceil - M'-\left \lceil \frac{-rM'}{r_1+1} \right \rceil = 1+ \left \lceil \frac{-rM'}{r_1+1} - \frac{r}{r_1+1} \right \rceil - \left \lceil \frac{-rM'}{r_1+1} \right \rceil \geq 0 $$ 
where the inequality follows since $r_1+1 \geq r$. Hence (\ref{sizeeq}) is an increasing function of $M$ as required.

For (iii) if $M=0$ then $|S| \leq \ceil{n(t-1)/t}$ as required. So assume $M \geq t$. Then by Corollary~\ref{Mu3} we have

\begin{align*}
|S| \leq & \, M - \left \lfloor \frac{rM}{r_2(p+q)} \right \rfloor -(r_1 r_2 -r_1-r_2+1) \left \lfloor \left \lfloor \frac{rM}{r_1(p+q)} - \frac{1}{r_2} \right \rfloor \frac{1}{r_1 r_2}\right \rfloor + \left \lceil \frac{(n-M)(t-1)}{t} \right \rceil \\
\leq & \, M - \frac{rM}{r_2(p+q)}+1-\frac{r_1 r_2 - r_1 -r_2+1}{r_1 r_2} \bigg( \frac{rM}{r_1(p+q)}-\frac{1}{r_2}-1 \bigg) +r_1 r_2-r_1-r_2+1 \\
& - \frac{M(t-1)}{t} + \left \lceil \frac{n(t-1)}{t} \right \rceil \\
\leq & \, \left \lceil \frac{n(t-1)}{t} \right \rceil + r_1 r_2-r_1-r_2+3 - M \bigg( \frac{r(r_1^2+(r_1-1)(r_2-1))}{tr_1^2 r_2(r_1+r_2)}-\frac{1}{t}\bigg) \\
= & \, \left \lceil \frac{n(t-1)}{t} \right \rceil + r_1 r_2-r_1-r_2+3 - M \bigg( \frac{r}{tr_2}\bigg(r_1+1+\frac{r_2-1}{r_1^2+(r_1-1)(r_2-1)}\bigg)^{-1} -\frac{1}{t}\bigg) \\
\leq & \, \left \lceil \frac{n(t-1)}{t} \right \rceil + r_1 r_2-r_1-r_2+3 - M \bigg( \frac{r_1 r_2-r_1-r_2+4}{t}-\frac{1}{t}\bigg) \\
\leq & \, \left \lceil \frac{n(t-1)}{t} \right \rceil,
\end{align*}
where the penultimate inequality follows by our lower bound on $r$ and the last inequality follows by using $M \geq t$.
\end{proof}

For Theorem~\ref{Mu4}(iii) it is easy to check that actually given the conditions on $r$ we must always have $t>1$ (we just state $t>1$ in the theorem for clarity). As an example, $p:=3t$, $q:=2t$, $r\geq 41$, and $t\geq r/2$ gives a set of equations which satisfy the conditions of Theorem~\ref{Mu4}(iii).

Theorem~\ref{Mu4} together with 
 Proposition~\ref{Mu1} yield results for $\mu_{\LL}(n)$ where $\LL$ is an equation with more than three variables. Full details can be found in~\cite{thesis}.

\begin{cor}\label{Mu5}
Let $\LL$ denote the equation $a_1 x_1 + \cdots + a_k x_k + b_1 y_1 + \cdots + b_{\ell} y_{\ell} = c_1 z_1+ \cdots + c_m z_m$ where the $a_i,b_i,c_i \in \mathbb{N}$ and $p':=\sum_i a_i$, $q':=\sum_i b_i$ and $r':=\sum_i c_i$ satisfy $p' \geq q' \geq r'$. Let $t':=\gcd(p',q',r')$ and write $p:=p'/t'$, $q:=q'/t'$ and $r:=r'/t'$. Let $t:=\gcd(p,q)$.
\begin{enumerate}[(i)]
\item{If $m=1$, $\ell=1$, $q'=b_1$ divides $a_i$ for all $1\leq i \leq k$ and $p+q \leq rq$ then $\mu_{\LL}(n)=\ceil{(q-1)n/q}$;}
\item{If $q$ divides $p$ and $p+q \geq rq$ then $\ceil{(p+q-r)n/(p+q)} \leq \mu_{\LL}(n) \leq \ceil{(p+q-r)(n-a)/(p+q)}+a$ where $a$ is the unique non-negative integer $0 \leq a < q$ such that $n-a$ is divisible by $q$;}
\item{Write $r_1:=p/t$ and $r_2:=q/t$. If $q$ does not divide $p$, $m=1$, $tt'$ divides $a_i$ and $b_j$ for $1 \leq i \leq k$, $1 \leq j \leq \ell$ and $$r>(r_1r_2-r_1-r_2+4)r_2\bigg(r_1+1+\frac{r_2-1}{r_1^2+(r_1-1)(r_2-1)}\bigg)$$ then $\mu_{\LL}(n)=\ceil{(t-1)n/t}$.}
\end{enumerate} \qed
\end{cor}

\section{The number of maximal solution-free sets}\label{2MSFS}
We start this section with the proof of Theorem~\ref{MainT1}.

\begin{thm-hand}[\ref{MainT1}.] 
{\it
Let $\LL$ denote the equation $px+qy=rz$ where $p \geq q \geq r$ and $p,q,r$ are fixed positive integers satisfying $\gcd(p,q,r)=1$. Let $t:=\gcd(p,q)$. Then $$f_{\max}(n,\LL) \leq 2^{\frac{Crn}{q}+o(n)} \text{ where } C:=1-\frac{t}{p+q}\Big(\frac{p^2+(p-t)(q-t)}{p^2} \Big).$$}
\end{thm-hand}

\begin{proof}
First note that $C$ lies between $1/2$ and $1-t/(p+q)$. To see this, note that if $q$ divides $p$, then $C=1-q/(p+q) \geq 1/2$ since $p \geq q$. Otherwise, $p>q>t$, and so $(p-t)(q-t)<p^2$. Hence $t(p^2+(p-t)(q-t))/(p^2(p+q)) < 2t/(p+q) \leq 2(q/2)/(p+q) < 1/2$ and so $C> 1/2$. We observe that $C \leq 1-t/(p+q)$ since $p \geq q \geq t$.

Let $\mathcal{F}$ denote the set of containers obtained by applying Lemma~\ref{L1Ext}. Since every maximal $\LL$-free subset of $[n]$ lies in at least one of the $2^{o(n)}$ containers, it suffices to show that every $F \in \mathcal F$ houses at most $2^{Crn/q+o(n)}$ maximal $\LL$-free sets. 

Let $F \in \mathcal F$. By Lemmas~\ref{L1Ext}(i) and~\ref{L2Ext}, $F=A \cup B$ where $|A|=o(n)$, $|B| \leq \mu_\LL(n)$ and $B$ is $\LL$-free. 
Define $M:=\max \{x \in B: x \equiv 0 $ mod $t\}$ and $u:=\max \{\floor{rM/q},\floor{rn/2q}\}$. Every maximal $\LL$-free set which lies in such a container can be constructed by:
\begin{enumerate}[(i)]
\item{Picking $S_1 \subseteq A$ to be $\LL$-free;}
\item{Adding a set $S_2 \subseteq [u] \cap B$ so that $S_1 \cup S_2$ is $\LL$-free;}
\item{Choosing a set $S_3 \subseteq [u+1,n] \cap B$ so that $S_1 \cup S_2 \cup S_3$ is a maximal $\LL$-free subset of $[n]$.}
\end{enumerate}

There are $2^{o(n)}$ ways to pick $S_1$. If $M \leq n/2$ then $u=\floor{rn/2q}$ and so there are at most $2^{rn/2q} \leq 2^{Crn/q}$ ways to pick $S_2$ so that $S_1 \cup S_2$ is $\LL$-free. Write $r_1:=p/t$ and $r_2:=q/t$. If $M \geq n/2$ then since $M$ is divisible by $t$, we apply Lemma~\ref{MainL1} to show that 
\begin{align*}
|[u] \cap B|= & \, |[\floor{rM/q}] \cap B| \\
\leq & \left \lfloor \frac{rM}{q} \right \rfloor - \left \lfloor \frac{rM}{r_2(p+q)} \right \rfloor -(r_1 r_2-r_1-r_2+1) \left \lfloor \left \lfloor \frac{rM}{r_1(p+q)} - \frac{1}{r_2} \right \rfloor \frac{1}{r_1 r_2} \right \rfloor \\
= & \, \frac{CrM}{q}+o(n).
\end{align*}
Hence there are at most $2^{CrM/q+o(n)} \leq 2^{Crn/q+o(n)}$ ways to pick $S_2$ so that $S_1 \cup S_2$ is $\LL$-free. 

Let $B':=[u+1,n]\cap B$. For step (iii) we calculate the number of extensions of $S_1 \cup S_2$ into $B'$. Observe by Lemma~\ref{L4Ext}, this is bounded above by MIS$(L_{S_1 \cup S_2}[B'])$. We will show that this link graph has only one maximal independent set. Then combining steps (i)-(iii) we have that $F$ contains at most $2^{o(n)} \times 2^{Crn/q+o(n)} = 2^{Crn/q+o(n)}$ maximal $\LL$-free sets as desired. 

If the link graph only contains loops and isolated vertices, then it has only one maximal independent set. For it to have an edge between distinct vertices, we either must have $x,z \in B'$, $y \in S_1 \cup S_2$ such that $px+qy=rz$ or $py+qx=rz$, or we must have $x,y \in B'$, $z \in S_1 \cup S_2$ such that $px+qy=rz$. 

The first of these events does not occur since otherwise $rz \geq q(x+y) > qx \geq q(\floor{rM/q}+1) >rM$ and so $z>M$. Note that since $z$ is part of the solution $px+qy=rz$ and $\gcd(p,q,r)=1$, it must be divisible by $t$. However this contradicts $z>M$ as we have $z \in B$, but $M$ was defined to be the largest element in $B$ divisible by $t$. 

If $M>n/2$ then the second event does not occur since $rz=px+qy \geq q(x+y) \geq 2q(\floor{rM/q}+1) > 2rM >rn$ and so $z>n$. If $M \leq n/2$ then the second event does not occur since $rz=px+qy \geq q(x+y) \geq 2q(\floor{rn/2q}+1) > rn$ and so again $z>n$.
\end{proof}

Note that when $r=1$, Theorem~\ref{MainT1} gives us new results for equations of the form $px+qy=z$. The authors previously obtained results for such equations in~\cite{HT1}. In the appendix we give a summary describing which result gives the best upper bound for various values of $p$ and $q$.

When $\LL$ denotes the equation $qx+qy=z$ for some positive integer $q\geq 2$, Proposition~26(iii)\COMMENT{AT:changes to paper 1... still Prop 26?} from~\cite{HT1} gives a lower bound of $f_{\max}(n,\LL) \geq 2^{(n-6q)/2q}$. Combining this with Theorem~\ref{MainT1} gives us the following asymptotically exact result.

\begin{thm-hand}[\ref{qxqyz}.] 
{\it
Let $\LL$ denote the equation $qx+qy=z$ where $q\geq 2$ is an integer. Then $$f_{\max}(n,\LL)=2^{n/2q+o(n)}.$$}
\end{thm-hand}

By adapting the proof of Theorem~\ref{MainT1} we obtain the following result for $f_{\max}(n,\LL)$ for some equations with more than three variables. 

\begin{thm}\label{MainT1Ext}
Let $\LL$ denote the equation $p_1 x_1 + \cdots + p_k x_k = rz$ where $p_1, \dots, p_k,r \in \mathbb{N}$ satisfy $\gcd(p_1,\dots,p_k,r)=1$ and $p_1 \geq \cdots \geq p_k \geq r$. Suppose that $p:=\sum_{i=1}^{k-1} p_i$ and $q:=p_k$ satisfy $t:=\gcd(p,q)=\gcd(p_1,\dots,p_k)$. Then $$f_{\max}(n,\LL) \leq 2^{\frac{Crn}{q} +o(n)} \text{ where } C:=1-\frac{t}{p+q}\Big(\frac{p^2+(p-t)(q-t)}{p^2} \Big).$$
\end{thm}

\begin{proof}
We follow the proof used in Theorem~\ref{MainT1} precisely (except for using Lemma~\ref{MainL1Ext} instead of Lemma~\ref{MainL1} in step (ii)) up until counting the number of ways of extending $S_1 \cup S_2$ to a maximal $\LL$-free set in $B':=[u+1,n]\cap B$. Observe by Lemma~\ref{L4Ext}, this is bounded above by MIS$(L_{S_1 \cup S_2}[B'])$ since $B'$ and $S_1 \cup S_2$ are $\LL$-free. To see that $B'$ is $\LL$-free, suppose $(x_1,\dots,x_k,z)$ is a solution within $B'$ and note that $rz=p_1 x_1 +\cdots + p_k x_k > p_k x_k \geq q(rM/q)=rM$ and so $z>M$. (Here we needed that each $p_i$ is positive.) Since $\gcd(p_1,\dots,p_k,r)=1$ and $\gcd(p,q)=\gcd(p_1,\dots,p_k)$ we have $\gcd(t,r)=1$ and so in any solution to $\LL$, $z$ must be divisible by $t$. However this contradicts $z>M$ as we have $z \in B$, but $M$ was defined to be the largest element in $B$ divisible by $t$.

We will show that this link hypergraph $L_{S_1 \cup S_2}[B']$ has only one maximal independent set (and hence the number of maximal $\LL$-free sets contained in $F$ is at most $2^{Crn/q+o(n)}$ as required). If the link hypergraph only contains loops and isolated vertices, then it has only one maximal independent set. 

For it to have a hyperedge between at least two vertices, there must exist a solution $(x_1,\dots,x_k,z)$ where either there is a hyperedge with distinct vertices $x_i,z \in B'$ for some $1 \leq i \leq k$ and $\{x_1,\dots,x_{i-1},x_{i+1},\dots,x_k\} \subseteq B' \cup S_1 \cup S_2$, or there is a hyperedge with distinct vertices $x_i,x_j \in B'$ for some $1\leq i<j \leq k$ and $\{x_1,\dots,x_{i-1},x_{i+1},\dots,x_{j-1},x_{j+1},\dots,x_k,z\} \subseteq B' \cup S_1 \cup S_2$.

Suppose the first event occurs with $(x_1,\dots,x_k,z)$. Then $rz=p_1 x_1+ \cdots+p_k x_k > p_i x_i \geq p_k x_i =q x_i \geq q(rM/q) =rM$ and so $z>M$. But since $z$ is part of a solution, it must be divisible by $t$. This contradicts $z \in B$, since $M$ was defined to be the largest element in $B$ divisible by $t$. 

If $M>n/2$ then the second event does not occur since $rz=p_1 x_1+ \cdots+p_k x_k > p_k (x_i+x_j) \geq 2q (\floor{rM/q}+1) >2rM >rn$ and so $z>n$. If $M \leq n/2$ then the second event does not occur since $rz=p_1 x_1+ \cdots+p_k x_k > p_k(x_i+x_j) \geq 2q (\floor{rn/(2q)}+1)>rn$ and so again $z>n$.
\end{proof}

We end the section with a lower bound.
\begin{prop}
Let $\LL$ denote the equation $qx+qy=rz$ where $q>r$ and $q,r$ are fixed positive integers satisfying $\gcd(q,r)=1$. Then $$f_{\max}(n,\LL) \geq 2^{\ceil{\floor{rn/2q-rq/2}(q-1)/q}-1}.$$
\end{prop}

\begin{proof}
Let $B$ be the $\LL$-free set $\{z \in [n]: z \not \equiv 0 $ mod $ q\}$. Let $M:=\max\{z \in [n]: rz/q^2 \in [n]\}$; so $M > n-q^2$. Let $S:=\{M\}$ and consider the link graph $L_S[B]$. Note that if $i \in B$ where $i<rM/q$ then $rM/q-i \in B$. This follows since $rM/q^2 \in \mathbb{N}$ and so $rM/q-i \not \equiv 0 $ mod $ q$. Hence there is an edge in $L_S[B]$ between every such $i$ and $rM/q-i$ since $q(i+rM/q-i)=rM$. By running through all $i \in B$ we obtain a total of $\ceil{\floor{rM/2q}(q-1)/q}$ disjoint edges in $L_{S}[B]$ of which at most one is a loop (at $rM/2q$ if it is an integer not congruent to 0 modulo $q$). Hence we obtain an \emph{induced} matching $E$ in $L_S[B]$ of size $\ceil{\floor{rM/2q}(q-1)/q}-1 \geq \ceil{\floor{rn/2q-rq/2}(q-1)/q}-1$. 
It is easy to see that the matching $E$ contains $2^{|E|}$ maximal independent sets. Since $E$ is an induced subgraph of $L_S[B]$, by applying Lemma~\ref{L6Ext} we obtain the result.
\end{proof}

\begin{ques}\label{q1}
Let $\LL$ denote the equation $qx+qy=rz$ where $q>r \geq 2$ and $q,r$ are fixed positive integers satisfying $\gcd(q,r)=1$. Does $f_{\max}(n,\LL)=2^{rn(q-1)/(2q^2)+o(n)}$?
\end{ques}

In the appendix just before Proposition~\ref{best2} we describe the value of $\mu^* _{\LL}(n)$. We remark that, using this, it turns out that for equations $\LL$ as in Question~\ref{q1}, we have that $2^{rn(q-1)/(2q^2)+o(n)}
=2^{(\mu_{\LL}(n)-\mu^*_{\LL}(n))/2+o(n)}$.

\section{Concluding Remarks}
In this paper we have used Lemma~\ref{GM1} as a tool to prove both results on the size of the largest solution-free subset of $[n]$ as well as on the number of maximal solution-free subsets of $[n]$.
Recall that Green~\cite{G-R} showed that given a homogeneous linear equation $\LL$, then $f(n, \LL)=2^{\mu_{\LL} (n)+o(n)}$. One can actually very easily apply Lemma~\ref{GM1} to obtain
that $f(n, \LL)=\Theta(2^{\mu_{\LL} (n)})$ for \emph{some} linear equations $\LL$ of the form $px+qy=rz$ where $p\geq q\geq r$ are positive integers. However, the results we obtain seem quite niche. Thus, we defer 
their statement and proof to the thesis of the first author~\cite{thesis}.

The crucial trick used in the proof of Theorems~\ref{MainT1} and~\ref{MainT1Ext} was to choose our sets $S$ carefully so that the link hypergraphs $L_S[B]$  each contain precisely one maximal independent set.
In other applications of this method~\cite{BLST, BLST2, HT1} the approach had been instead to obtain other structural properties of the link graphs (such as being triangle-free) to ensure there are not too
many maximal independent sets in $L_S[B]$. It would be interesting to see if the approach of our paper can be applied to obtain other (exact) results in the area.


Although we have found an initial bound on $f_{\max}(n,\LL)$ for some equations with more than three variables, we still do not know in general if there are significantly fewer maximal $\LL$-free subsets of $[n]$ than there are $\LL$-free subsets of $[n]$. 
Progress on giving general upper bounds on the number of maximal independent sets in (non-uniform) hypergraphs should (through the method of link hypergraphs) yield results in this direction.

\section*{Acknowledgments}
The authors are grateful to the referee for a careful review which particularly aided the quality of the exposition in the introduction.

{\footnotesize \obeylines \parindent=0pt

Robert Hancock, Andrew Treglown
School of Mathematics
University of Birmingham
Edgbaston
Birmingham
B15 2TT
UK
}
\begin{flushleft}
{\emph{E-mail addresses}:
\tt{\{rah410,a.c.treglown\}@bham.ac.uk}}
\end{flushleft}

\begin{appendix}
\section{}
In this appendix we give a summary of the best known upper bound on $f_{\max}(n,\LL)$ for equations of the form $px+qy=rz$ where $p \geq q \geq r$. First we recall some of the results from~\cite{HT1}.

\begin{thm}\label{max1}\cite{HT1}
Let $\LL$ be a fixed homogeneous three-variable linear equation. Then $$f_{\max}(n,\LL) \leq 3^{(\mu_{\LL}(n)-\mu_{\LL}^*(n))/3+o(n)}.$$
\end{thm}

\begin{thm}\label{max2}\cite{HT1}
Let $\LL$ denote the equation $px+qy=z$ where $p \geq q \geq 2$ are integers so that $p \leq q^2-q$ and $\gcd (p,q)=q$.  Then $$f_{\max}(n,\LL) \leq 2^{(\mu_{\LL}(n)-\mu_{\LL}^*(n))/2+o(n)}.$$
\end{thm}

\begin{thm}\label{max3}\cite{HT1}
Let $\LL$ denote the equation $px+qy=rz$ where $p\geq q \geq r$ and $p,q,r \in \mathbb{N}$. Then $$f_{\max}(n,\LL) \leq 2^{\mu_{\LL}(rn/q)+o(n)}.$$ 
\end{thm}

\begin{prop}\label{best}\cite{HT1}
Let $\LL$ denote the equation $px+qy=z$ where $p\geq q$, $p\geq 2$ and $p,q \in \mathbb{N}$. 
The best upper bound on $f_{\max}(n,\LL)$ given by Theorems~\ref{max1},~\ref{max2} and~\ref{max3}  is:

\begin{enumerate}[(i)]
\item{$f_{\max}(n,\LL) \leq 3^{(\mu_{\LL}(n)-\mu_{\LL}^*(n))/3+o(n)}$ if $\gcd(p,q)=q$, $p\geq q^2$, and either $q\leq 9$ or $10\leq q \leq 17$ and $p<(a-1)(q^2-q)/(q(2-a)-1)$ where $a:=\log_3(8)$;}
\item{$f_{\max}(n,\LL) \leq 2^{(\mu_{\LL}(n)-\mu_{\LL}^*(n))/2+o(n)}$ if $\gcd(p,q)=q$ and $p \leq q^2-q$;}
\item{$f_{\max}(n,\LL) \leq 2^{\mu_{\LL}(n/q)+o(n)}$ otherwise.}
\end{enumerate}
\end{prop}

Recall $\mu^*_{\LL}(n)$ denotes the number of elements in $[n]$ which do not lie in \emph{any} non-trivial solution to $\LL$ that only consists of elements from $[n]$. 
Let $\LL$ denote the equation $px+qy=rz$ where $p\geq q \geq r$ and $p,q,r$ are fixed positive integers satisfying $\gcd(p,q,r)=1$.
Let $t:=\gcd(p,q)$. Then notice that $S:=\{s \in [n]: s>\floor{(rn-p)/q}, t \nmid s\}$ is a set of elements which do not lie in any solution to $\LL$ in $[n]$. This follows since if $s> \floor{(rn-p)/q}$ then $ps+q \geq qs+p > rn$ and so $s$ cannot play the role of $x$ or $y$ in an $\LL$-triple in $[n]$. If $t \nmid s$ then as $t|(px+qy)$ for any $x,y \in[n]$ but $\gcd(r,t)=1$ we have that $s$ cannot play the role of $z$ in an $\LL$-triple in $[n]$. Actually, for large enough $n$, every element that does not lie in any solution to $\LL$ in $[n]$ is in $S$, and so we have $\mu^*_{\LL}(n)=|S|=\ceil{(n-\floor{(rn-p)/q})(t-1)/t}$ for all such $\LL$. We omit the proof of this here.

We can now state the new summary.
\begin{prop}\label{best2}
Let $\LL$ denote the equation $px+qy=rz$ where $p\geq q \geq r$, $p\geq 2$ and $p,q,r$ are fixed positive integers satisfying $\gcd(p,q,r)=1$.
Let $t:=\gcd(p,q)$ and let $a:=\log_2 3$. The best upper bound on $f_{\max}(n,\LL)$ given by Theorems~\ref{max1},~\ref{max2},~\ref{max3} and~\ref{MainT1} is:
\begin{enumerate}[(i)]
\item{$f_{\max}(n,\LL) \leq 3^{(\mu_{\LL}(n)-\mu_{\LL}^{*}(n))/3+o(n)}$ if 
\begin{enumerate}[(a)]
\item{$r=1$, $\gcd(p,q)=q$, $p\geq \max{\{q^2,(q^2-q)a/(q(3-2a)+a)\}}$, and $q \leq 9$;}
\item{$r\geq 2$, $\mu_{\LL}(n)=\ceil{(t-1)n/t}$, and additionally (1) $p \neq q$ or (2) $2 \leq q\leq 18$;}
\item{$r \geq 2$, $q$ divides $p$, $p+q \geq rq$ and additionally (1) $p \neq q$ or (2) $2 \leq q\leq 18$;}
\end{enumerate}}
\item{$f_{\max}(n,\LL) \leq 2^{Crn/q+o(n)}$ where $C:=1-t(p^2+(p-t)(q-t))/(p^2(p+q))$ if
\begin{enumerate}[(a)]
\item{$r=1$, $\gcd(p,q) \neq q$ or $q>9$ or $p<q^2$ or $p<(q^2-q)a/(q(3-2a)+a)$;}
\item{$r\geq 2$, $\mu_{\LL}(n)=\ceil{(t-1)n/t}$, and $p=q \geq 19$.}
\end{enumerate}}
\end{enumerate}
\end{prop}

\begin{proof}
First suppose that $r=1$ (and so $\mu_{\LL}(n)=\ceil{(p+q-1)n/(p+q)}$). Note that $C \leq 1-t/(p+q)=(p+q-t)/(p+q) \leq (p+q-1)/(p+q)$ and so the exponent given by Theorem~\ref{MainT1} is at most the exponent given by Theorem~\ref{max3}. For Theorem~\ref{max2} we require $\gcd(p,q)=q$ and $p \leq q^2-q$. In this case $C=p/(p+q)$ and $(p+q-1)/(2(p+q))-(q-1)^2/(2q^2)=(2pq+q^2-p-q)/(2q^2(p+q)) \geq p/(q(p+q))=C/q$ and so the exponent given by Theorem~\ref{MainT1} is at most the exponent given by Theorem~\ref{max2}.
 
It remains to check when the bound given by Theorem~\ref{max1} is still better than the bound given by Theorem~\ref{MainT1}. By Proposition~\ref{best} this can only possibly be the case if $\gcd(p,q)=q$ and $p\geq q^2$. To prove (i)(a) it suffices to show that $$3^{\frac{(p+q-1)n}{3(p+q)}-\frac{(q-1)^2 n}{3q^2}} \leq 2^{\frac{pn}{q(p+q)}},$$ or rearranging $$p(q(3-2a)+a)\geq (q^2-q)a.$$ If $q\geq 10$ then $(q(3-2a)+a)$ is negative, but then we would require $p$ negative, a contradiction. Hence we must have $q \leq 9$ and then the inequality holds if $p>(q^2-q)a/(q(3-2a)+a)$.

Now suppose that $r \geq 2$ and $\mu_{\LL}(n)=\ceil{(t-1)n/t}$. Then $\mu_{\LL}(n)-\mu^*_{\LL}(n)=r(t-1)n/(qt)+o(n)$ and $3^{x/3}<2^{x}$ and so Theorem~\ref{max1} gives a better bound than Theorem~\ref{max3}. We wish to know when $$3^{\frac{r}{q} \frac{t-1}{3t}} < 2^{\frac{r}{q} (1-\frac{t}{p+q}(\frac{p^2+(p-t)(q-t)}{p^2}))}.$$ Write $r_1:=p/t$ and $r_2:=q/t$. The above rearranges to give $$t((a-3)r_1^2(r_1+r_2)+3r_1^2+3(r_1-1)(r_2-1))<ar_1^2(r_1+r_2).$$ The right hand side is positive and the left hand side is negative unless $r_1=r_2=1$. In this case $p=q=t$ and so we now require $3^{(q-1)/(3q)}<2^{1/2}$, which holds when $q\leq 18$. 

Finally suppose that $r \geq 2$, $q$ divides $p$ and $p+q \geq rq$ (so $\mu_{\LL}(n)=\ceil{(p+q-r)n/(p+q)}$). Since $q$ divides $p$, we have $t=q$ and $p=r_1 q$, and so Theorem~\ref{MainT1} gives a bound of $2^{rpn/(q(p+q))+o(n)}$. This is better than Theorem~\ref{max3} which gives a bound of $2^{r(p+q-r)n/(q(p+q))+o(n)}$ since $q \geq r$. Therefore we wish to know when $$3^{\frac{p+q-r}{3(p+q)} - \frac{(q-r)(q-1)}{3q^2}} < 2^{\frac{rp}{q(p+q)}}.$$ Rearranging, we require $r_1(a(q+qr-r)/3-qr) \leq a(r-q)/3$. Now note $a(q+qr-r)/3-qr$ is negative when $r \geq 2$, so this rearranges to give $r_1 \geq (q-r)/(r-q-rq+3rq/a)$. If $p>q$ (so $r_1 \geq 2$), it suffices to have $2 \geq (q-r)/(r-q-rq+3rq/a)$ or rearranging, $q(r(2-6/a)+3) \leq 3r$. This holds since $r(2-6/a)+3$ is negative for $r \geq 2$. Otherwise $p=q$, and so since $p+q\geq rq$, we have that $r=2$. So we require $1 \geq (q-2)/((6/a-3)q+2)$ which holds when $q \leq 18$. (In this final case, $\mu_{\LL}(n)=\ceil{(t-1)n/t}=\ceil{(p+q-r)n/(p+q)}=\ceil{(q-1)n/q}$.)



\end{proof}
\end{appendix}
\end{document}